		\def\version{19 January, 2023}				       % 
\documentclass[reqno,12pt]{amsart} 
\usepackage{amsmath,amsthm,amssymb,a4,graphics,color}
\usepackage[active]{srcltx}
\usepackage[usenames,dvipsnames]{pstricks}
\usepackage{epsfig}
\usepackage{pst-grad}
\usepackage{pst-plot} 
\usepackage{tikz}
\usepackage{hyperref}

\newfam\Bbbfam 
\font\tenBbb=msbm10 
\font\sevenBbb=msbm7 
\font\fiveBbb=msbm5 
\textfont\Bbbfam=\tenBbb 
\scriptfont\Bbbfam=\sevenBbb 
\scriptscriptfont\Bbbfam=\fiveBbb

\newtheorem{theorem}{Theorem}[section] 
\newtheorem{lemma}[theorem]{Lemma} 
 
\newtheorem{cor}[theorem]  {Corollary} 
\newtheorem{remark}[theorem]  {Remark}

\theoremstyle{definition}

\newcommand{\E}{\mathbb E}

\newcommand{\R}{\mathbb R}
\newcommand{\N}{\mathbb N}

\renewcommand{\P}{\mathbb P}

\newcommand{\smfrac}[2]{\textstyle{\frac {#1}{#2}}}
\def\1{{\mathchoice {1\mskip-4mu\mathrm l}      % Blackboard bold 1 
{1\mskip-4mu\mathrm l} 
{1\mskip-4.5mu\mathrm l} {1\mskip-5mu\mathrm l}}} 
\newcommand{\ssup}[1] {{\scriptscriptstyle{({#1})}}}

\renewcommand{\qed}{\hfill\ensuremath{\square}}

\renewcommand{\d}{{\rm d}} 
 
\newcommand{\eps}{\varepsilon} 
 
\newcommand{\Bin}{{\rm Bin}} 
\newcommand{\Poi}{{\rm Poi}} 
 
\newcommand{\id}{{\rm{id}}}

\newcommand{\Mcal}   {{\mathcal M }}

\newcommand{\Xcal}   {{\mathcal X }}

\newcommand{\e}   {{\operatorname e }}

\numberwithin{equation}{section}

\begin{document}
\title[Throughput in (slotted) ALOHA]{The throughput in multi-channel (slotted) ALOHA:\\\medskip  large deviations and analysis of bad events}
\author[Wolfgang K\"onig and Charles Kwofie]{}
\maketitle
\thispagestyle{empty}
\vspace{-0.5cm}

\centerline{\sc 
Wolfgang K\"onig\footnote{TU Berlin and WIAS Berlin, Mohrenstra{\ss}e 39, 10117 Berlin, Germany, {\tt koenig@wias-berlin.de}} 
and Charles Kwofie\footnote{University of Energy and Natural Resources, Department of Mathematics and Statistics, Post Office Box 214, Sunyani, Ghana,  {\tt charles.kwofie@uenr.edu.gh}}}
\renewcommand{\thefootnote}{}
\vspace{0.5cm}
\centerline{\textit{WIAS Berlin and TU Berlin, and University of Energy and Natural Resources, Sunyani }}

\bigskip

\centerline{\small(\version)} 
\vspace{.5cm}

\begin{abstract}
We consider ALOHA and slotted ALOHA protocols as medium access rules for a multi-channel message delivery system. Users decide randomly and independently with a minimal amount of knowledge about the system at random times to make a message emission attempt. We consider the two cases that the system has a fixed number of independent available channels, and that interference constraints make the delivery of too many messages at a time impossible. 

We derive probabilistic formulas for the most important quantities like the number of successfully delivered messages and the number of emission attempts, and we derive large-deviation principles for these quantities in the limit of many participants and many emission attempts. We analyse the rate functions and their minimizers and derive laws of large numbers for the throughput. We optimize it over the probability parameter. Furthermore, we are interested in questions like ``if the number of successfully delivered messages is significantly lower than the expectation, was the reason that too many or too few sending attempts were made?''. Our main tools are basic tools from probability and the theory of (the probabilities of) large deviations.
 \end{abstract}

\vspace{.2cm}


\section{Introduction and main results}

\subsection{Introduction}\label{sec-Intro}

Protocols for medium access control (MAC) are fundamental and ubiquitous in any telecommunication system. Here we are particularly interested in {\em multi-channel systems}, where a fixed number of channels is available. In order to keep the complexity of the algorithm of the channel choices by the transmitters low, we make a well-known probabilistic ansatz and assume that each transmitter chooses randomly and independently a channel for each transmission. This makes the system get along with a minimum of infrastructure, i.e, with a minimum knowledge about the occupancy of the channels. In other words, we consider an ALOHA-based multi-channel protocol, see \cite{RS90}. More specifically, we concentrate in this paper on {\em slotted ALOHA}, where message transmissions are possible only in specific micro time slots.

It is our purpose to study random events that comprise the transmission of many messages from many transmitters in a large number of (very short) time-slots, forming a fixed time interval, in the limit of many such slots. In each of the slots, each transmitter chooses with a certain probability, independently over all transmitters and over all slots, whether to make a transmission attempt in that slot or not. This probability must be very small, i.e., on the scale of the inverse of the number of transmitters. This leads to a huge number of random decisions that have to be drawn in each time slot,
with a tiny probability each, which leads to a huge amount of data with high imprecision. 

In this paper, we give a probabilistic analysis of the main quantities, like numbers of attempts and of successes, per micro time slot in the limit of many such time slots, coupled with many message emission attempts. In particular, we comprehensively characterize the main quality parameter, the {\em throughput}. We are going to find neat descriptions of the entire (joint) distributions of these quantities and of their limits. In particular, we introduce techniques from the probabilistic theory of (the probabilities of) large deviations. Using this theory, we analyse events that have a very low probability in this limit, like the event that the number of successes is significantly lower than its expectation. Furthermore, we give an explicit assessment of the most likely reason for this. In this way, we go far beyond calculating (limiting) expectations, but we handle the numbers of message attempts and transmission successes per slot as stochastic processes with a rich structure.

In our system, we have a fixed upper bound $\kappa$ for the number of messages that can be successfully delivered in a given micro time slot. Our main system parameter is the probability parameter $p$, the {\it medium access probability (MAP)}, with which each of the messages tries randomly to gain access to the system. If $p$ is too large, then it is likely that the system exceeds the upper bound $\kappa$, which results into failures of many message transmissions. On the other hand, if $p$ is too small, then a part of the possible capacity is not exhausted, and the system underachieves. One of our goals is to quantify an optimal choice of $p$. The main quantity for this criterion is the {\it throughput}, the number of successfully transmitted messages per time unit. But we analyse also other quantities like the number of message attempts.

In the {\em multi-channel (MC) models} that we consider in this paper, we assume a total interference isolation between the channels, i.e., we neglect possible interferences between them. Here each channel in a given micro time slot is able to successfully transmit one message, if no more than one emission attempt is made through this channel. The higher the number of emission attempts is, the higher is the number of sucesses (but also the number of unsuccessful messages, which we could also analyse with our ansatz, but abstained from); hence an optimization over the probability parameter is only of limited interest, unless there is a substantial price that is paid per unsuccessful transmission.

Closely related to multi-channel systems are systems with entirely unlimited interference between all of them. Here the success of the transmission of the messages is regulated by means of the {\em signal-to-interference ratio (SIR)}. In a simplified setting, the transmission of message $i$ in a given time slot is successful if and only if
$$
\frac{1}{\sum_{j\in I\setminus\{i\}} 1}\geq \tau,
$$
where $\tau\in(0, \infty)$ is a technical constant, and $I$ is the index set of messages that attempt to transmit in this slot (which depends on various quantities, like the number of message emission attempts in that slot, which may be random). Since we are working in a spaceless model, there is no distance and therefore no path-loss function involved, and we give the same signal strength power $1$ to each transmission attempt. Putting $\kappa =1+\lfloor \frac 1\tau\rfloor\in\N$, we see that any transmission attempt in the slot is successful if and only if no more than $ \kappa $ attempts are made in the slot; otherwise interference makes all these attempts unsuccessful. This is the second of the two model functionalities that we are going to study; we call it an {\em interference-based (IB)} model. Mathematically, it shows great similarities to multi-channel models, but the most important difference is that a high number of emission attempts leads to many unsuccessful attempts and is therefore working against a high throughput; hence an optimization over the probability parameter is of high interest and not an easy task.

While the derivation of the expected throughput in the multi-channel ALOHA model and its optimization over $p$ is easy (with the well-known result that the maximal throughput is equal to $\kappa/\e$ with $\kappa$ the number of channels), for the interference-based model, we can offer an explicit formula for the expectation, but only approximate characterisations of the maximization over $p$, which get sharp in the limit as $\kappa\to\infty$. 

We would like to point out
that, from a mathematical-practical point of view, it might have advantages to let each transmitter decide, for the entire time interval under consideration, whether or not an attempt is made during that interval, and then to randomly and uniformly  distribute the attempts over the time slots of this interval. We call the first mode of attempt decisions {\em local} and the latter {\em global}. We will be studying both in this paper, since we believe that both have their right and their advantages. On the level of expectations, there will be no difference noticeable between the main quantities of interest, but in the large-deviation behavior.

Summarizing, the main new contributions of the present paper are the following.

\begin{enumerate}
	\item[(a)] describing the relevant quantities in terms of their entire
	joint distribution (rather than only expectations),
	\item[(b)] describing limiting events of large deviations asymptotically
	in terms of explicit rate functions,
	\item[(c)] comparing local and global random assignments of transmission
	slots,
	\item[(d)] optimizing the throughput over the MAP for the
	interference-based model,
	\item[(e)] analysis of large deviation probabilities of conditional
	events (e.g., of a low number of successes).
\end{enumerate}

The remainder of this paper is organized as follows. We introduce our models in Section~\ref{sec-slottedALOHA} and the most important quantities and questions in Section~\ref{sec-Quantities}. Our results are presented and commented in Section~\ref{sec-Results}, and some comments on the literature are made in Section~\ref{sec-literature}. 
Section~\ref{sec-ProofsLDP} brings all the proofs of the large-deviation principles, and Section~\ref{sec-Ratefunctions} the proofs of the other results.

\subsection{Description of the models}\label{sec-slottedALOHA}

Let us introduce the models that we are going to analyse.  We  consider a reference time interval, which we pick as $[0,1]$. We have a large parameter $N\in\N$, which models a large number of network participants and a large number of time slots. The reference time interval is divided in to $N$ slots $[\frac{i-1}N,\frac iN)$ for $i\in[N]=\{1,\dots,N\}$; every message delivery starts at the beginning of one of these slots and terminates before its elapsure. With a fixed parameter $b\in(0,\infty)$, we assume that $bN$ participants (we waive the integer-part brackets) are in the system, i.e., at any time $bN$ devices would like to emit one message each. Access to the medium is under some random rule, for which we consider two variants, a rule that is {\it local in time} and one that is {\it global in time}; both have a parameter $p\in(0,\infty)$.

\medskip\noindent{\bf Access rules:}

\begin{itemize}
\item[(L)] Under the {\em local rule} each of the $bN$ participants chooses at any time slot randomly with probability $\frac pN$ to emit a message during this slot, independently over all $bN$ participants and all $N$ time slots.

\item [(G)] Under the {\em global rule}  each of the $bN$ participants chooses randomly with probability $p$ whether to emit a message during some of the $N$ time slots, and then all those who choose that option are randomly and uniformly distributed over the $N$ time slots.
\end{itemize}

Under Rule (G), any participant has only at most one chance during $[0,1]$, while under Rule (L), every message has an unbounded number of trials and can be successful several times uring $[0,1]$. Hence, under (G), $p$ needs to be in $(0,1]$, while under (L), it can be any positive number, assuming that $N$ is large (and we assume this). We assume that each participant has an unbounded number of packages to be sent, i.e., it makes successively an unbounded number of emission attempts. Rule (G) has a two-step random strategy, as first each message randomly decides whether to attempt a transmission, and then picks randomly a microscopic time slot. Here the number of random variables that need to be sampled is much smaller than under Rule (L), and the probability parameter is of finite order in $N$, in contrast to Rule (L). We therefore see substantial practical advantages in Rule (G) over Rule (L).

Now we describe the criteria for successful delivery of the messages that are choosen to be emitted under either Rule (L) or (G). We consider two scenarios, the {\it multi-channel scenario} and the  {\it interference-based scenario}; both come with a parameter $\kappa\in\N$: 

\medskip\noindent{\bf Success rules:}
\begin{itemize}
\item [(MC)] In the {\em multi-channel scenario}, the are $\kappa$ channels available, and in each slot each of the emission attempts choose randomly and uniformly one of the $ \kappa$ channels, independent over all the other participants and time slots. A transmission attempt is successful in this slot if no other participant chooses the channel that it picked. All other attempts are unsuccessful.

\item [(IB)] In the {\em interference-based scenario}, in any given time slot,  all transmission attempts are successful if their number does not exceed $\kappa$; otherwise all attempts in that slot are unsuccessful.

\end{itemize}

In the case of a successful attempt of transmission of a message, we say that the participant has gained access to the medium. As we explained in Section~\ref{sec-Intro}, Scenario (MC) describes slotted ALOHA with $\kappa$ channels and total absence of infrastructure, while Scenario (IB) describes the influence of interference constraints. Note that Model (B) in \cite{HLS12} is contained in Scenario (MC).

We are going to couple each of the two scenarios (MC) and (IB) with each of the two Rules (L) and (G) and obtain four different protocols. Scenario (MC), coupled with Rule (L), is equal to Model (B) in \cite{HLS12}.

\subsection{Quantities and questions of interest}\label{sec-Quantities}

There are three parameters in our simple models:
\begin{itemize}
\item[$\bullet$] $p\in(0,\infty)$ the emission attempt probability parameter,

\item[$\bullet$] $b\in(0,\infty)$ the rate of messages that would like to be transmittted during $[0,1]$,

\item[$\bullet$] $\kappa\in\N$ the threshold for the success criterion.
\end{itemize}

We consider $\kappa$ (given by technical conditions) and $b$ (given by the appearance of participants) as given quantities that cannot be controled. However, the parameter $p$ can be picked by the system operator and can be adapted to $b$ and $\kappa$; it is decisive for the success of the system. Part of our investigations will be on an optimal choice of $p$ given $\kappa$ and $b$.

The quantities that we are interested in are the following.
\begin{itemize}
\item $A_N=$ the number of message sending attempts,

\item $S_N=$ number of successfully sent messages,

\item (only for Secnario (IB)) $R_N=$  number of successful slots, that is, slots in which all messages are successfully transmitted.
\end{itemize}

These three quantities are defined on probability spaces whose probability measures are denoted by $\P^{\ssup N}_{\rm D,E}$ with ${\rm D}\in\{{\rm L,G}\}$ and ${\rm E}\in\{{\rm MC,  IB}\}$, respectively.

The most important quantity is the {\it throughput}, the number of succcessfully sent messages per time unit, which is equal to $S_N/N$ in our model. But we find it also important to consider the number of unsuccessful sending attempts, in order to be able to say something about the frustration of the participants of the system.

In both scenarios, in order to maximize the number of successes, one would like to pick the probability parameter $p$ in such a way that the expected number of transmission attempts per slot is close to $\kappa$, i.e., $p\approx \kappa /b$. However, if the number of attempts fluctuates upwards, then the success  is damaged, in (IB) even maximally damaged; hence the optimal choice of $p$ should be a bit lower. Part of our analysis is devoted to finding the optimal value of this parameter.

\subsection{Our results}\label{sec-Results}

In this section we describe and comment on our results: Section~\ref{sec-LDPs} on large-deviations, Section~\ref{sec-LLN} on laws of large numbers, Section~\ref{sec-optimalp} on the optimal choice of the probability parameter $p$, and Section~\ref{sec-successes} on the question where the event of having few successes  most likely comes from.

We denote the Poisson distribution with parameter $\alpha\in(0,\infty)$ on $\N_0$ by $\Poi_\alpha=(\e^{-\alpha}\frac{\alpha^k}{k!})_{k\in\N_0}$, and the binomial distribution on $\{0,1,\dots,N\}$ with parameters $N\in\N$ and $p\in(0,1)$ by $\Bin_{N,p}(k)=\binom Nk p^k(1-p)^{N-k}$. Furthermore, we denote the entropy of a probability measure $\mu$ on some discrete set $\mathcal X$ with respect to another one, $\nu$, by $H(\mu|\nu)=\sum_{k\in\Xcal}\mu_k\log\frac{\mu_k}{\nu_k}$. Recall that $\mu\mapsto H(\mu|\nu)$ is non-negative, strictly convex and is zero only for $\mu=\nu$. By $\Mcal_1(\mathcal X)$ we denote the set of probability measures on $\Xcal$.

\subsubsection{Large-deviation principles}\label{sec-LDPs}

Our first main result is on the asymptotics as $N\to\infty$ of the joint distribution of $(S_N,A_N,R_N)$, in the sense of a large-deviation principle. First we turn to (IB).

\begin{theorem}[LDP for $\frac 1N(A_N,S_N,R_N)$ for Scenario (IB)]\label{thm-LDPALOHA}
Fix the model parameters $b,p>0$ and $\kappa\in\N$, where we assume $p\leq 1$ for $\rm {D=G}$. Then for both ${\rm D}\in\{{\rm L,G}\}$, the tuple $\frac 1N(A_N,S_N,R_N)$ satisfies a large-deviation principle (LDP) under $\P^{\ssup N}_{{\rm D,IB}}$ with rate function given by
\begin{equation}\label{Idefmicro}
I_{\rm{L,IB}}(a,s,r)=\inf\Big\{H(\mu|{\rm Poi}_{bp})\colon\mu\in\Mcal_1(\N_0),\sum_{k\in\N_0} f(k)\mu_k=(a,s,r)\Big\}
\end{equation}
where $f(k)=(k, k\1\{k\le\kappa\},\1\{k\le\kappa\})$, while
\begin{equation}\label{Idef}
I_{\rm {G,IB}}(a,s,r)=I_{\rm {L,IB}}(a,s,r)+(b-a)\log\frac{1-\frac ab}{1-p}+a-bp.
\end{equation}
\end{theorem}

The proof is in Section~\ref{sec-ProofLDPmicroslot} for Rule (L) and in Section~\ref{sec-ProofLDPmicroslot-GI} for Rule (G). An alternate proof under Rule (L) is described in Section~\ref{sec-ProofThm1}; this leads to a very different formula for the rate function.

The stated LDP says that for any open, respectively closed, set $G,F\subset [0,b]\times[0,b]\times[0,1]$ ,
\begin{eqnarray*}
\limsup_{N\to\infty}\frac{1}{N}\log\P^{\ssup N}_{{\rm D,IB}} \Big(\frac 1N\big(A_N, S_{N},R_N\big)\in F\Big)&\leq&-\underset{F}{\inf}I_{{\rm D,IB}},\\
\underset{N\rightarrow\infty}{\lim\inf}\frac{1}{N}\log\P^{\ssup N}_{{\rm D,IB}}\Big(\frac 1N\big(A_N, S_{N},R_N\big)\in G\Big)&\geq&-\underset{G}{\inf}I_{{\rm D,IB}}.
\end{eqnarray*}
This can be symbolically summarized by saying that for any $(a,s,r)$
$$
\P^{\ssup N}_{\rm{D,IB}}\big({\smfrac 1N}(A_N,S_N,R_N)\approx (a,s,r)\big)\approx \e^{-N I_{\rm {D,IB}}(a,s,r)},\qquad N\to\infty.
$$
See \cite{DZ10} for an account on the theory of (the probabilities of) large deviations.

\begin{remark}[LDP for $S_N$]\label{rem-LDPSN}
A standard corollary of Theorem~\ref{thm-LDPALOHA}
is an LDP for the number $S_N$  of successes, which follows directly from the contraction principle (which says that $(\varphi(X_N))_{N\in\N}$ satisfies an LDP if $(X_N)_{N\in\N}$ does and $\varphi$ is continuous, and it gives a formula for the rate function). Indeed $\frac 1N S_N$ satisfies an LDP under $\P_{\rm {D,IB}}^{\ssup N}$ with rate function for D$=$L
$$
s\mapsto \inf_{a,r}I_{\rm {L,IB}}(a,s,r)=\inf\Big\{H(\mu|\Poi_{bp})\colon\mu\in\Mcal_1(\N_0),\sum_{k\in[\kappa]} k\mu(k)=s\Big\}.
$$
This formula is further analysed as a by-product in the proof of Theorem~\ref{thm-conditionalthroughput}. A conclusion is that the probability to have less than $N (s_{\rm IB}(p,\kappa)-\eps)$ successes decays exponentially fast with rate $\inf\{H(\mu|\Poi_{bp})\colon\mu\in\Mcal_1(\N_0),\sum_{k\in[\kappa]} k\mu(k)\leq s_{\rm IB}(p,\kappa)-\eps\}$, which is a positive number. Certainly, the analogous statement holds also for Rule (G). Furthermore, we can also apply the contraction principle to obtain an LDP for $R_N$ or for the pair $(A_N,S_N)$.
\hfill$\Diamond$
\end{remark}

\begin{remark}[Higher precision]\label{rem-highprec} 
With more of technical work, we could also prove the following, stronger assertion.
Fix $a,s\in[0,b]$ satisfying $s\leq a$ and fix $r\in[0,1]$. Pick sequences $a_N,s_N,r_N\in\frac 1N \N_0$ such that $a_N\to a$, $s_N\to s$ and $r_N\to r$ as $N\to \infty$. Then for ${\rm D}\in\{\rm {G, L}\}$,
\begin{equation}\label{LDPlimit}
I_{\rm {D,IB}}(a,s,r)=-\lim_{N\to\infty}\frac 1N\log \P^{\ssup N}_{\rm{D,IB}}\big(A_N=N a_N,S_N=Ns_N, R_N=N r_N\big).
\end{equation}
\hfill$\Diamond$
\end{remark}

\begin{remark}[Difference of the rate functions]\label{rem-diffratefcts} In the proof in Section~\ref{sec-ProofLDPmicroslot-GI} it will turn out that, under Rule (L), $A_N$ has the distribution of $N$ independent $\Bin_{bN,p/N}$-distributed random variables, while unter Rule (G), $A_N$ is $\Bin_{bN,p}$-distributed. Given $A_N$, the distribution of $(S_N,R_N)$ is the same under both rules. The last term on  the right-hand side of \eqref{Idef} (i.e., the difference of the two rate functions) is equal to the difference of the two rate functions for $\frac 1N A_N$. These two rate functions are
\begin{eqnarray}
J_{\rm{L}}(a)&=&pb-a+a\log\frac{a}{pb},\\
J_{\rm{G}}(a)&=&a\log\frac{a}{p}+(b-a)\log\frac{b-a}{1-p}-b\log b,
\end{eqnarray}
and the last term in \eqref{Idef} is equal to $J_{\rm{G}}(a)-J_{\rm{L}}(a)$. Note that
$$
J_{\rm{G}}'(a)=\log\frac a{b-a}+\log\frac {1-p}p,\qquad J_{\rm{G}}''(a)=\frac b{a(b-a)},
$$
and $J_{\rm{L}}'(a)= \log\frac a{bp}$ and $J_{\rm{L}}''(a)= \frac 1{bp}$. Hence, $J_{\rm{L}}''(bp)<J_{\rm{G}}''(bp)$ and therefore, for $a$ in a neighbourhood of the minimal site $bp$ outside $bp$, we see that $J_{\rm{L}}(a)<J_{\rm{G}}(a)$. This shows that under Rule (G) the number of attempts has a smaller variance (even on the exponential scale) than under Rule (L), which we consider as a structural advantage of (G) over (L).
\hfill$\Diamond$
\end{remark}

\begin{remark}[Analysis of rate function]\label{rem-Anarate fct}
On the first view, the formula in \eqref{Idefmicro} seems to be rather involved, but in the proof of Theorem~\ref{thm-conditionalthroughput} we will find the minimizing $\mu$ for $\inf_r I_{\rm L,IB}(a,s,r)$ and will characterize it using standard variational analysis.
\hfill$\Diamond$
\end{remark}

\begin{remark}[Alternative rate function]\label{rem-altratefct}
Our proof of Theorem~\ref{thm-LDPALOHA} in Sections~\ref{sec-ProofLDPmicroslot} and \ref{sec-ProofLDPmicroslot-GI} is based on Sanov's theorem and the contraction principle and leads to an entropy description of the rate function.  In Section~\ref{sec-ProofThm1} we give an alternate proof of Theorem~\ref{thm-LDPALOHA} using Cram\'er's theorem, leading to a representation of the rate function involving Legendre transforms of logarithms of moment-generating functions. This representation appears in \eqref{ratefctalternate}.
\hfill$\Diamond$
\end{remark}

%%%%%%%%%%%%%%%%%%%%%%%%%%%%%%%%%%%%%%%%%%%%%%%%%%%%%%%%%%%%%%%%%%5

Now we turn to our LDP for the multi-channel case. Recall that Model (B) in \cite{HLS12} is contained in what we called Scenario (MC).

\begin{theorem}[LDP for $\frac 1N(A_N,S_N)$ for Scenario (MC)]\label{thm-LDPMultichannel}
	Fix the model parameters $b,p>0$ and $\kappa\in\N$ channels, where we assume $p\leq 1$ for Rule $\rm {D=G}$. Then the tuple $\frac 1N(A_N,S_N)$ satisfies an LDP under $\P^{\ssup N }_{\rm{ D,MC}}$ for ${\rm D}\in\{{\rm L,G}\}$ with rate function (for $\rm {D=L}$)
		\begin{equation}\label{ILMC}
I_{\rm{L,MC}}(a,s)=\inf\Big\{H(\nu|M)\colon\nu\in\Mcal_1(\Xi),\sum_{(i,j)\in \Xi}\nu(i,j)i=a,\sum_{(i,j)\in\Xi}\nu(i,j)j=s \Big\},
	\end{equation}
	where $\Xi=\{(i,j)\in\N_0^2\colon j\leq i\mbox{ and }j\leq \kappa\}$ and the reference probability measure $M$ on $\Xi$ is given as
\begin{equation}\label{Rdef}
M(i,j)=\Poi_{bp/\kappa}^{\otimes \kappa}\Big(\sum_{k\in[\kappa]} X_k=i,\sum_{k\in[\kappa]} \1_{\{X_k=1\}}=j\Big),
\end{equation}
and, for $\rm {D=G}$, with rate function given as
	\begin{equation}\label{IGMC}
	\begin{aligned}
	I_{\rm{G,MC}}(a,s)&=\kappa\inf\Big\{H(\mu|{\rm Poi}_{bp/\kappa})\colon\mu\in\Mcal_1(\N_0),\sum_{g\in\N_0}\mu(g)g=\frac{a}{\kappa},\mu(\{1\})=\frac{s}{\kappa} \Big\}\\
	&\quad+a-bp+(b-a)\log\frac{1-\frac ab}{1-p}.
	\end{aligned}
	\end{equation}
\end{theorem}

The proof is in Section~\ref{sec-ProofLDPmicroslot-LMC} for Rule (L) and in Section~\ref{sec-ProofLDPmicroslot-GMC} for Rule (G).

\begin{remark}[Interpretation]\label{rem-Interpr}
The reference measure $M$ has the interpretation of a channel-choice distribution. Indeed, the Poisson-distributed variables $X_k$, $k \in[\kappa]$, with parameter $bp$ stand for the number of participants that choose the $k$-th channel for the transmission attempt; then $M(i,j)$ is the probability that in total $i$ attempts are made and $j$ successes are earned.
\hfill$\Diamond$
\end{remark}

\begin{remark}[Contraction principle]\label{rem-ContrPrinc}
The analogous assertions of Remark~\ref{rem-LDPSN} about an LDP for $S_N$, e.g., hold certainly also for Scenario (MC).
\hfill$\Diamond$
\end{remark}

\begin{remark}[Difference of the two rate functions]\label{rem-diffratefctMC}
The difference of the two rate functions in \eqref{IGMC} is the same as in \eqref{Idef}, but the reason is different from the reason in Scenario (IB) (see Remark~\ref{rem-diffratefcts}). It comes out by some explicit manipulation of the distribution of $(A_N,S_N)$, for which cannot offer an easy interpretation.
\hfill$\Diamond$
\end{remark}

\begin{remark} Like for Scenario (IB), we could prove, with more technical work, the following also in Scenario (MC). Fix $a,s\in[0,b]$ satisfying $s\leq a$ and pick sequences $a_N,s_N\in\frac 1N \N_0$ such that $a_N\to a$ and $s_N\to s$ as $N\to \infty$. Then for $D\in\{L,G\}$,
	\begin{equation}\label{LDPlimit-Multichannel-G}
	I_{\rm{D,MC}}(a,s)=-\lim_{N\to\infty}\frac 1N\log \P^{\ssup N}_{\rm{D,MC}}\big(A_N=N a_N,S_N=Ns_N\big)
	\end{equation}
	\end{remark}

\subsubsection{Laws of large numbers}\label{sec-LLN}

It is a standard conclusion from the LDP that, if the rate function has a unique minimizer at $(a_p,s_p,r_p)$, a law of large numbers (LLN) follows, i.e., $\frac 1N(A_N, S_{N},R_N)\to (a_p,s_p,r_p)$ in  probability with exponential decay of the probability of being outside a neighbourhood of $(a_p,s_p,r_p)$. Hence, the following statement implies two LLNs.

\begin{cor}[LLN for the throughput in Scenario (IB)]\label{cor-LLNSALOHA}
The  two rate functions $I_{\rm{G,IB}}$ and $I_{\rm{L,IB}}$ are both strictly convex and possess the same unique minimizer $(a_{\rm IB}(p,\kappa),s_{\rm IB}(p,\kappa),r_{\rm IB}(p,\kappa))$ given by
\begin{eqnarray}
a_{\rm IB}(p,\kappa)&=&pb=\E_{\Poi_{bp}}(X),\\
s_{\rm IB}(p,\kappa)&=&\e^{-bp}\sum_{i=0}^\kappa i \frac{(bp)^i}{i!}=\E_{{\rm Poi}_{bp}}[X\1\{X\leq \kappa\}]=bp \,\e^{-bp}\sum_{i=0}^{\kappa-1}\frac{(bp)^i}{i!},\label{minimizers}\\
r_{\rm IB}(p,\kappa)&=&\e^{-bp}\sum_{i=0}^{\kappa}\frac{(bp)^i}{i!}={\rm Poi}_{bp}([0,\kappa]).\label{minimizerr}
\end{eqnarray}
\end{cor} 

\begin{proof} Just recall that the map $\mu\mapsto H(\mu|\Poi_{bp})$ is strictly convex and has the unique minimizer $\mu=\Poi_{bp}$; hence the unique minimizing $(a,s,r)$ must be compatible with that, i.e., equal to $\sum_{k\in\N_0} f(k)\Poi_{bp}(k)$.
\end{proof}
In particular, the throughput in Scenario (IB) is equal to the $\Poi_{bp}$-expectation of $X\1\{X\leq \kappa\}$, and the typical rate of successful micro time slots is ${\rm Poi}_{bp}([0,\kappa])$.

In the same way, we see the analogous statement for (MC):

\begin{cor}[LLN for the throughput in Scenario (MC)]\label{cor-LLNSALOHA-MC}
The two rate functions $I_{\rm{G,MC}}$ and $I_{\rm{L,MC}}$ are both strictly convex and possess the same unique minimizer 
$$
\big(a_{\rm MC}(p,\kappa),s_{\rm MC}(p,\kappa)\big)=\big(pb, pb\e^{-bp/\kappa}\big).
$$
\end{cor}
In particular, the throughput in Scenario (MC) is equal to $bp\e^{-bp/\kappa}$. 

\subsubsection{Optimal $p$}\label{sec-optimalp}

A natural and important question is about that value of $p$ that maximizes the expected throughput per micro slot, $s_{\rm IB}(p,\kappa)$, respectively $s_{\rm MC}(p,\kappa)$. Since $p$ is restricted to $[0,1]$ under  Rule (G), we will consider only Rule (L), where we can optimize over all $p\in(0,\infty)$.

For Scenario (MC), the answer is easily derived by differentiating: the optimal $p$ is equal to $\kappa/b$, and the optimal throughput is equal to $\kappa/\e$.

Scenario (IB) is more interesting. It is clear that the optimal value of $p$ should be such that $bp$ is smaller than $\kappa$, since otherwise the number of attempts per time slot is larger than the success threshold. But the question is how much below one should go in order not to underachieve more than necessary.

\begin{lemma}[Optimal $p$]\label{lem-optimalp}
For any $\kappa\in\N$, there is precisely one $p_*\in(0,\infty)$ that maximizes the map $(0,\infty)\ni p\mapsto s_{\rm IB}(p,\kappa)$. It is  characterised by
\begin{equation}\label{pMax}
\frac {(a_*)^\kappa}{(\kappa-1)!}=\sum_{i\in\N_0\colon i\leq \kappa-1}\frac{(a_*)^i}{i!},\qquad a_*=b p_*,
\end{equation}
and it satisfies $b p_*<\kappa-1$ and $b p_*\sim \kappa$ as $\kappa\to\infty$. More precisely, we even have $b p_*\geq (\kappa-\sqrt \kappa)^{1-\kappa^{-1/2}}$ for any $\kappa$. Furthermore, $p\mapsto s_{\rm IB}(p,\kappa)$ strictly increases in $[0,p_*]$ and strictly decreases in $[p_*,\infty)$.
\end{lemma}

The proof of Lemma~\ref{lem-optimalp} is in Section~\ref{sec-optimisep}.

\subsubsection{Conditioning the number of attempts on the number of successes}\label{sec-successes}
In this section we discuss an interesting question in the interferenced-based scenario, where too many messages lead to a serious descrease of throughput:  what is the most likely reason for a deviation event of the form that the throughput is below the theoretically optimal one? Have there been too many message emission attempts, such that the interference canceled many, or did the system underachieve, i.e., had fewer attempts than could be handled successfully?

This question can be answered with the help of large-deviation theory, combined with an analysis of the rate functions. We handle this only for the Rule (L), where we can work with any value of $p\in(0,\infty)$. In order to formalize this question, we write $\P^{\ssup {N,p}}_{\rm L,IB}=\P^{\ssup {N,p}} $ for the probability measure in Scenario (IB) with parameter $p$ and $\E^{\ssup {N,p}} $ for the corresponding expectation. Picking some $0<s\leq a$, then it follows from Remark~\ref{rem-highprec} that
$$
\lim_{N\to\infty}\frac 1N\log \P^{\ssup {N,p}} \big(A_N=\lfloor aN\rfloor\,\big|\,S_N=\lfloor N s\rfloor\big)=-\inf_{r} I_{\rm L,IB}^{\ssup p}(a,s,r)+\inf_{\widetilde a,r}I_{\rm L,IB}^{\ssup p}(\widetilde a,s,r),
$$
where we wrote $I_{\rm L,IB}^{\ssup p}$ for the rate function $I_{\rm L,IB}$ defined in \eqref{Idefmicro}.
From this, we see that
$$
\lim_{N\to\infty}\E^{\ssup {N,p}} \Big(\frac{A_N}{N}\Big| S_N=\lfloor Ns\rfloor\Big)=\underset{a}{\mathrm{argmin}}\Big(\inf_r I_{\rm L,IB}^{\ssup p}(a,s,r)\Big).
$$
(The latter can also be derived from Theorem~\ref{thm-LDPALOHA} instead from the unproved Remark~\ref{rem-highprec}.)
Given $s$, we now define $a_p(s)$ as a minimizer of the map $a\mapsto\inf_r I_{\rm L,IB}^{\ssup p}(a,s,r)$, i.e., the typical rate of sending attempts, conditional on having $\approx s N$ successes. It will turn out that $a_p(s)$ is well-defined at least in a neighbourhood of $a_p(s_p)$ if $p$ is close enough to $p_*=p_*({\rm L,IB})$, where we now abbreviate $s_p=s_{\rm L,IB}(p,\kappa)$ for the minimizer that we established in Corollary \ref{cor-LLNSALOHA}, and $p^*$ is the maximizing $p$ for $(0,\infty)\ni p\mapsto s_p$ characterized by \eqref{pMax}. In terms of these quantities, the question now reads: Given $s<s_p$ , is it true that $a_p(s)<a_p(s_p)$?

\begin{theorem}\label{thm-conditionalthroughput} Fix $\kappa$. Then, for any $p\in(0,\infty)$ and for any  $s$ in some neighbourhood of $s_p$, we have
\begin{eqnarray}
p< p^*\quad&\Longrightarrow&\quad\Big[s<s_p\Rightarrow a_p(s)<a_p(s_p)\Big]\quad\mbox{and}\quad \Big[s>s_p\Rightarrow a_p(s)>a_p(s_p)\Big]\label{psmall},
\\
p>p^*\quad&\Longrightarrow&\quad\Big[s<s_p\Rightarrow a_p(s)>a_p(s_p)\Big]\quad\mbox{and}\quad \Big[s>s_p\Rightarrow a_p(s)<a_p(s_p)\Big].
\label{plarge}
\end{eqnarray}
Furthermore, for $p=p_*$, for any $s\in[0,p_*b]\setminus \{s_{p_*}\}$, we have  $a_{p^*}(s)>a_{p^*}(s_{p^*})$.
\end{theorem}

The proof is in Section \ref{sec-fewsuccesses}. Theorem \ref{thm-conditionalthroughput} says that, for non-optimal $p$, if $s$ sufficiently close to the optimal $s_p$, then the attempt number $a_p(s)$ deviates to the same side of $a_p(s_p)$ as $s$ is with respect to $s_p$, while in the optimal $p_*$, the typical attempt number for non-optimal success number is always larger than the optimal one. The latter means that, for the optimal choice $p=p_*$, the event of non-optimal throughput alway comes with overwhelming probability from too many attempts. Apparently, here the conditional probability for having too many attempts is much larger than the one for having too few.

\subsection{Literature remarks}\label{sec-literature}

A wide range of multiple access protocols have been extensively discussed in the literature; see for example \cite{RS90, BG92, lakatos, throughputevaluation}.
See \cite{HLS16, Y91, TTH18, I11} for an explanation of the advantages and disadvantages of multi-channel ALOHA protocols from a operational point of view and a description of transmit-reference modulation (TR Modulation) for handling the problem of synchronizing simultaneous message transmissions in such systems. \cite{HLS12} gives some probabilistic analysis of a few concrete ALOHA variants, but fails to give tractable formulas; Model (B) there is identical to our Scenario (MC) under Rule (L). In \cite{C20,C20b}, additional functionalities are investigated as a possible improvement of the throughput  by means of an additional exploration phase.

A systematic probabilistic analysis of the performance of ALOHA protocols has been started for the single-channel pure ALOHA in the 1950s; see 
\cite{abramson, performance} and some of the above mentioned references. The throughput is identified there as $\lambda \e^{-2\lambda}$, which also coincides with our result for $s_{\rm ALOHA}(\lambda,1)$ in the special case $\kappa=1$. In \cite{performance}, \cite{lakatos}, \cite{RS90} and \cite{throughputevaluation} one can also read about the more popular and better known single-channel version of ALOHA, namely the {\em slotted ALOHA}, which offers the higher throughput $\lambda \e^{-\lambda}$. The multi-channel case of this model has also been studied, e.g., in \cite{shen}, where the throughput $\lambda \e^{-\lambda/\kappa}$ has been calculated. In the present paper, we re-derive this value and combine it with a large-deviation analysis with explicit rate functions. 

To the best of our knowledge, in {\it continuous} time there are no results for the multi-channel model in the literature yet that are similar to those of the present paper, with the recent exception \cite{KS22}, where the ALOHA and the {\it Carrier Sense Multiple Access (CSMA) protocol} are analysed and similar results are derived as in the present paper for slotted ALOHA in discrete time. The difference is that in the interference constraint is valid in any fixed time interval, but not only in all the determined micro time slots. Hence, \cite{KS22} does not find a description in terms of independent random variables, but in terms of a Markov renewal process.

\section{Proofs of the LDPs}\label{sec-ProofsLDP}

\subsection{Proof of Theorem \ref{thm-LDPALOHA} for Rule (L)}\label{sec-ProofLDPmicroslot}

In this section, we prove the LDP for Scenario (IB) under Rule (L). Recall that we write $[k]=\{1,\dots,k\}$ for $k\in\N$.

For $i\in [bN]$ and $j\in[N]$, we let $X^{\ssup j}_i\in\{0,1\}$  be the indicator on the event that the $i$-th participant chooses to attempt to send a message in the $j$-th time slot. All these random variables are independent Bernoulli random variables with parameter $p/N$. Let 
$$
A^{\ssup j}_N=\sum_{i\in[bN]}\1\{X^{\ssup j}_i=1\},\qquad R^{\ssup j}_N=\1\{A^{\ssup j}_N\leq \kappa\},\qquad S^{\ssup j}_{N}=A^{\ssup j}_N\1\{A^{\ssup j}_N\leq \kappa\}.
$$
Then $A_N^{\ssup j}$ is the number of transmission attempts, $R_N^{\ssup j}$ the indicator on the event that the $j$-th micro slot is successful and $S_N^{\ssup j}$ is the number of successfully sent messages during that time slot. Clearly, $A^{\ssup j}_N$ is binomially distributed with parameters $bN$ and $p/N$, and the collection of them over $j\in[N]$ is independent. Furthermore, $A_N=\sum_{j=1}^{N}A^{\ssup j}_N$, $R_N=\sum_{j=1}^{N}R^{\ssup j}_N$ and $S_N=\sum_{j=1}^{N}S^{\ssup j}_N$.  We introduce the empirical measure
$$
\mu_N:=\frac{1}{N}\sum_{j=1}^{N}\delta_{A^{\ssup j}_N},
$$
which is a random member of the set $\Mcal_1(\N_0)$ of probability measures on $\N_0$. Furthermore, we introduce
$$
f\colon \N_0\to\N_0\times [\kappa]\times\{0,1\},\qquad f(a)=\big(a, a\1\{a\leq \kappa\},\1\{a\leq \kappa\}\big).
$$
Note that the triple under interest, $(A_N, R_N,S_N)$, is nothing but $N\langle f,\mu_N\rangle$, i.e., the image of $\mu_N$ under the map $\mu\mapsto\langle f,\mu\rangle$. 

We abbreviate $q_k={\rm Poi}_{bp}(k)=\e^{-pb}(pb)^k/k!$ and $q=(q_k)_{k\in\N_0}$. If $A_N^{\ssup j}$ would be exactly ${\rm Poi}_{bp}$-distributed, then Sanov's theorem would  imply that $(\mu_N)_{N\in\N }$ satisfies an LDP with rate function $\mu\mapsto H(\mu|{\rm Poi}_{bp})$. Let us assume for a moment that $\langle f,\widetilde\mu_N\rangle$ satisfies an LDP with rate function given in \eqref{Idefmicro} if $\widetilde \mu_N$ is the empirical measure of independent ${\rm Poi}_{bp}$-distributed random variables $A^{\ssup 1},\dots,A^{\ssup N}$. We show that $\langle f,\mu_N\rangle$ and $\langle f,\widetilde\mu_N\rangle$ are exponentially equivalent as $N\to\infty$ and therefore satisfy the same LDP, namely the LDP of Theorem \ref{thm-LDPALOHA} for D$=$L with rate function  given in \eqref{Idefmicro}. For this, it suffices to show that, for a suitable coupling of the $A^{\ssup 1},\dots,A^{\ssup N}$ with the $A^{\ssup 1}_N,\dots,A^{\ssup N}_N$,
\begin{equation}\label{expequiv}
\lim_{N\to\infty}\frac 1N\log \P\Big(\sum_{j=1}^N |A_N^{\ssup j}-A^{\ssup j}|>\eps N\Big)=-\infty,\qquad \eps>0,
\end{equation}
since the second and third components of $f$ are smaller than the first one. We will show this for any coupling of these variables such that $\lim_{N\to\infty}\P(A_N^{\ssup 1}\not=A^{\ssup 1})=0$. We use Markov's inequality (or the exponential Chebyshev inequality) and the independence, to estimate, for any $C>0$,
$$
\P\Big(\sum_{j=1}^N |A_N^{\ssup j}-A^{\ssup j}|>\eps N\Big)
\leq \e^{-C\eps N}\E\Big[\e^{C |A_N^{\ssup 1}-A^{\ssup 1}|}\Big]^N.
$$
We are finished as soon as we have shown that $\lim_{N\to\infty}\E[\e^{C |A_N^{\ssup 1}-A^{\ssup 1}|}]=1$ for any $C>0$. In the expectation, we estimate $1\leq \1\{A_N^{\ssup 1}\leq K,A^{\ssup 1}\leq K\}+\1\{A^{\ssup 1}_N>K\}+\1\{A^{\ssup 1}>K\}$ and use once more the exponential Chebyshev inequality and then H\"{o}lder's inequality to obtain, for any $L>0$,
$$
\begin{aligned}
\E[\e^{C |A_N^{\ssup 1}-A^{\ssup 1}|}]&
\leq 1+\e^{2CK}\P(A_N^{\ssup 1}\not=A^{\ssup 1})\\
&\qquad+\e^{-L K}\Big(\sqrt{\E[\e^{2(C+L)A_N^{\ssup 1}}]\E[\e^{2C A^{\ssup 1}}]}+\sqrt{\E[\e^{2(C+L)A^{\ssup 1}}]\E[\e^{2C A_N^{\ssup 1}}]}\Big)\\
&\to 1+2\e^{-L K} \e^{bp(\e^{2(C+L)}-1)} \e^{bp(\e^{2C}-1)},\qquad N\to\infty,
\end{aligned}
$$
as an explicit calculation for the exponential moments of $A_N^{\ssup 1}$ and $A^{\ssup 1}$ shows. We pick now $L=1$ and make $K\to\infty$ to see that the right-hand side converges to one, which concludes the proof of \eqref{expequiv}.

It remains to show that $\langle f,\widetilde\mu_N\rangle$ satisfies an LDP with rate function given in \eqref{Idefmicro}. Then Sanov's theorem implies that $(\widetilde\mu_N)_{N\in\N}$ satisfies an LDP on $\Mcal_1(\N_0)$ with rate function $\mu\mapsto H(\mu|\Poi_{bp})$. 
If the map  $\mu\mapsto\langle f,\mu\rangle$  would be continuous in the weak topology on $\Mcal_1(\N_0)$, then the contraction principle immediately would give the assertion. However, clearly $f$ is not bounded, hence the map $\mu\mapsto\langle f,\mu\rangle$ is not continuous in the weak topology on $\Mcal_1(\N_0)$. Hence we cannot directly apply the contraction principle. Clearly, the second and third argument in the function are bounded. A sufficient cutting argument for the first argument is  given by proving that
\begin{equation}\label{cutting}
\lim_{C\to\infty}\limsup_{N\to\infty}\frac 1N\log \P_N\Big(\sum_{j=1}^N A^{\ssup j}>CN\Big)=-\infty.
\end{equation}
A proof of \eqref{cutting} is easily  derived using the exponential Chebyshev inequality as above and that $A^{\ssup 1},\dots,A_N^{\ssup N}$ are independent $\Poi_{bp}$-distributed random variables and that $\E[\e^{C A^{\ssup 1}}]= \e^{pb(\e^C-1)}$ for any $C$. Hence, modulo elementary technical details, the proof of Theorem \ref{thm-LDPALOHA} for Rule (L) follows from this.

\subsection{Proof of Theorem \ref{thm-LDPALOHA} under Rule (G)}\label{sec-ProofLDPmicroslot-GI}

In this section, we prove the LDP for Scenario (IB) under Rule (G).

We want to identify the large deviation behaviour of the probability distribution of the triple $(A_N,S_N,R_N)$ under $\P_{\rm {G,IB}}^{\ssup N}$. We have it already under $\P_{\rm {L,IB}}^{\ssup N}$. We are going to identify the former distribution now explicitly in terms of the latter.

For any $a, s,r\in\N_0$ we have the following;
\begin{equation}
\begin{aligned}
d_N&=\P_{\rm {G,IB}}^{(N)}\big(A_N= a,S_N=s, R_N= r\big)
\\&=\P_{\rm {G,IB}}^{\ssup N}(A_N=a)\P_{\rm {G,IB}}^{\ssup N}(S_N=s, R= r|A_N=a)
\\&=\P_{\rm {L,IB}}^{\ssup N}(A_N= a)\P_{\rm {L,IB}}^{\ssup N}(S_N=s, R_N=r|A_N=a)\frac{\P_{\rm {G,IB}}^{\ssup N}(A_N=a)}{\P_{\rm {L,IB}}^{\ssup N}(A_N=a)},
\end{aligned}
\end{equation}
where we used that $\P_{\rm {G,IB}}^{\ssup N}(S_N=s, R_N= r|A_N= a)=\P_{\rm {L,IB}}^{\ssup N}(S_N=s, R_N=r|A_N= a)$, since the success rules are the same for the local and the global access rules. Hence
\begin{equation}\label{SGI}
d_N=\P_{\rm {L,IB}}^{\ssup N}(A_N=a,S_N=s, R_N=r)\frac{\P_{\rm {G,IB}}^{\ssup N}(A_N=a)}{\P_{\rm {L,IB}}^{\ssup N}(A_N=a)}.
\end{equation}
Hence, the two rate functions $I_{\rm L,IB}$ and $I_{\rm G,IB}$ differ only by the exponential rate of the quotient. The latter is easily identified. Indeed, observe that $A_N$ is $\Bin_{bN,p}$ distributed under $\P^{\ssup N}_{\rm {G,IB}}$, hence, if $a_N\in\frac 1N \N_0$ satisfies $a_N\to a$, then Stirling's formula ($N!=(N/\e)^N\e^{o(N)}$ for $N\to\infty$) shows that
\begin{equation}
J_{\rm{G}}(a):=-\lim_{N\to\infty}\frac 1N\log \P^{\ssup N}_{\rm G,IB}\big(A_N=N a_N\big)=a\log\frac{a}{p}+(b-a)\log\frac{b-a}{1-p}-b\log b.
\end{equation}
Furthermore, under $\P^{\ssup N}_{\rm {L,IB}}$, $A_N$ is distributed as the sum of $N$ independent $\Bin_{bN,p/N}$-distributed random variables. 
We showed in Section~\ref{sec-ProofLDPmicroslot} (see \eqref{expequiv}) that $A_N$ is exponentially equivalent with a sum of $N$ independent $\Poi_{bp}$-distritbuted random variables, hence $A_N$ satisfies an LDP with the same rate function, more precisely,
\begin{equation}
J_{\rm{L}}(a):=-\lim_{N\to\infty}\frac 1N\log \P^{\ssup N}_{\rm L,IB}\big(A_N=N a_N\big)= pb-a+a\log\frac{a}{pb}.
\end{equation}
Hence, $\frac 1N(A_N,S_N,R_N)$  under $\P^{\ssup N}_{\rm {G,IB}}$ satisfies an LDP with rate function 
$$
I_{\rm {G,IB}}(a,s,r)=I_{\rm {L,IB}}(a,s,r)-J_{\rm G}(a)+J_{\rm L}(a),
$$ 
and this is equal to right hand side of \eqref{Idef}.

\subsection{Alternate proof of Theorem \ref{thm-LDPALOHA} under Rule (L)}\label{sec-ProofThm1} 

In this section, we indicate an alternative proof of the LDP of Theorem \ref{thm-LDPALOHA} in Scenario (IB) under Rule (L) with an alternate representation of the rate function that is very different from \eqref{Idefmicro}; see \eqref{ratefctalternate}. Indeed, it does not involve any entropy, but is instead based on formulas that appear in connection with Cram\'er's theorem, i.e., Legendre transforms of the logarithm of moment generating functions.

We use the notation of Section~\ref{sec-ProofLDPmicroslot}. Recall that $A_N^{\ssup j}$ is the number of emission attempts in the $j$-th micro time slot, $(\frac{j-1}N,\frac jN]$. Then $A_N^{\ssup 1},\dots,A_N^{\ssup N}$ are i.i.d., and each of them is $\Bin_{bN,p/N}$-distributed. Fix $a,s,r\in\N_0$ and consider the event $\{A_N=a,S_N=s,R_N=r\}$. This is the event that in precisely $r$ time slots the corresponding $A_N^{\ssup j}$ is $\leq \kappa$ (these time slots are successful) and in all the other $N-r$ time slots it is $>\kappa$ (these slots are unsuccessful), and that the total sum of all the $A_N^{\ssup j}$ with $A_N^{\ssup j}$ is equal to $s$. By permutation symmetry of the time slots, we may assume that all the first $r$ time slots are successful and the remainning ones are not. The total number of distinctions of the $N$ slots into $r$ successful and $N-r$ unsuccessful ones is $\binom Nr$. Hence, by independence of the $A_N^{\ssup j}$'s and after relabeling, we have
\begin{equation}\label{d1}
\begin{aligned}
\P^{\ssup N}_{\rm L,IB}&(A_N=a,S_N=s,R_N=r)\\
&=\binom{N}{r}\P\Big(A^{\ssup j}_{N}\le\kappa \;\forall j\in[r], \sum_{j\in[r]}A^{\ssup j}_{N}=s\Big)\\
&\qquad\times\P\Big(A^{\ssup j}_{N}>\kappa \;\forall j\in[N-r], \sum_{j\in[N-r]}A^{\ssup j}_{N}=a-s\Big)\\
&=\binom{N}{r}{\rm Bin}_{bN,p/N}([0,\kappa])^{r}\,{\tt P}^{\ssup N}_{\le\kappa}\Big(\frac{1}{r}\sum_{j\in[r]}A^{\ssup j}_{N}=\frac{s}{r}\Big)\\
&\qquad \times {\rm Bin}_{bN,p/N}((\kappa, \infty))^{N-r}\,{\tt P}^{\ssup N}_{>\kappa}\Big(\frac{1}{N-r}\sum_{j\in[N-r]}A^{\ssup j}_{N}=\frac{a-s}{N-r}\Big),
\end{aligned}
\end{equation}
where ${\tt P}^{\ssup N}_{\leq \kappa}$ is the expectation with respect to independent ${\rm Bin}_{bN,p/N}$-distributed variables, conditioned on being $\leq \kappa$, and ${\tt P}^{\ssup N}_{>\kappa}$ is defined analogously.

Now the remainder of the proof is clear. We replace $a,s,r\in\N$ by $a_N N, s_N N, r_NN \in\N$ with $a_N\to a$, $s_N\to s $ and $r_N\to r$ for some $a,s,r\in(0,\infty)$ and we find easily the large-$N$ exponential asymptotics of the binomial term and the two probability powers, and for the two probabilities involving the sums of $A_N^{\ssup j}$'s, we can use Cram\'er's theorem. Here are some details: We again use the Poisson limit theorem to see that ${\rm Bin}_{bN,p/N}([0,\kappa])^{r_NN}={\rm Poi}_{pb}([0,\kappa])^{rN}\e^{o(N)}$ as $N\to\infty$ and the analogous statement for the other probability term. Furthermore, we leave to the reader to check that the average of the $A_N^{\ssup j}$ under ${\tt P}^{\ssup N}_{\leq \kappa}$ satisfy the same LDP as the average of independent ${\rm Poi}_{bp}$-distributed random variables, conditioned on being $\leq \kappa$ and analogously with $>\kappa$ instead of $\leq \kappa$. (This is implied by a variant the exponential equivalence that we proved in Section~\ref{sec-ProofLDPmicroslot}: see \eqref{expequiv}.) The latter do satisfy an LDP, according to Cram\'er's theorem, with rate function equal to the Legendre transform of $y\mapsto \log {\tt E}_{\leq \kappa}[\e^{ y X_1}]$, where ${\tt E}_{\leq \kappa}$ is the expectation with respect to ${\tt P}_{\leq \kappa}$, and $X_1$ is a corresponding random variable. Hence we have that $\frac{1}{r_N N}\sum_{j\in[r_N N]}A^{\ssup j}_{N}$ satisfies an LDP under ${\tt P}^{\ssup N}_{\le\kappa}$ on the scale $N$ with rate function
$$
x\mapsto =r J_{\leq \kappa}(x),\qquad\mbox{where}\qquad J_{\leq \kappa}(x)=\sup_{y\in\R}\Big(x y-\log {\tt E}_{\leq \kappa}[\e^{ y X_1}]\Big),
$$
and an analogous assertion for the other probability term (last line of \eqref{d1}).
Note that Stirling's formula gives that $-\lim_{N\to\infty}\frac 1N\log\binom N{r_N N}=r\log r +(1-r)\log(1-r)$. Substitution all this in  the last two lines of \eqref{d1}, we obtain that $\frac 1N(A_N,S_N,R_N)$ satisfies under Rule (L) in Scenario (IB)  an LDP on the scale $N$ with rate function equal to
$$
\begin{aligned}
\widetilde I_{\rm L,IB}(a,s,r)&=r\log r +(1-r)\log(1-r)
+r J_{\leq \kappa}({\smfrac sr})-r\log \Poi_{bp}([0,\kappa])\\
&\qquad +(1-r) J_{>\kappa}({\smfrac {a-s}{1-r}})-(1-r)\log \Poi_{bp}((\kappa,\infty)).
\end{aligned}
$$
This can be rewritten as follows. Introducing $I_{\leq \kappa}(x)=\sup_{z\in\R}(xz-\log\sum_{i=0}^\kappa\e^{zi}/i!)$, we see, after making the substitution $\e^z=bp\e^y$, i.e., $y=z-\log(pb),$ that
$$
r J_{\leq \kappa}({\smfrac sr})-r\log \Poi_{bp}([0,\kappa])=rbp-s\log(bp)+r I_{\leq\kappa}({\smfrac s r}),
$$
and an analogous formula for the last term, resulting in
\begin{equation}\label{ratefctalternate}
\widetilde I_{\rm L,IB}(a,s,r)=r I_{\leq \kappa}({\smfrac sr})+(1-r) I_{>\kappa}({\smfrac {a-s}{1-r}})+bp-a\log(bp)+r\log r +(1-r)\log(1-r).
\end{equation}
Certainly, this function must coincide with $I_{\rm L,IB}$ defined in \eqref{Idefmicro}, but this is admittedly hard to see.

\subsection{Proof of Theorem \ref{thm-LDPMultichannel} under Rule (L)}\label{sec-ProofLDPmicroslot-LMC}

We are now proving the LDP of Theorem  \ref{thm-LDPMultichannel} in Scenario (MC) under the Rule (L). We recall some of the notation from Section \ref{sec-ProofLDPmicroslot}: for $i\in [bN]$ and $j\in[N]$, we let $X^{\ssup j}_i\in\{0,1\}$  be the indicator on the event that the $i$-th participant chooses to attempt to send a message in the $j$-th time slot. All these random variables are independent Bernoulli random variables with parameter $p/N$. Let $A^{\ssup j}_N=\sum_{i\in[bN]}\1\{X^{\ssup j}_i=1\}$, then $A_N^{\ssup j}$ is the number of transmission attempts. Clearly, $A^{\ssup j}_N$ is binomially distributed with parameters $bN$ and $p/N$, and the collection of them over $j\in[N]$ is independent. Furthermore, $A_N=\sum_{j=1}^{N}A^{\ssup j}_N$. 

Let us identify the distribution of the number $S_N^{\ssup j}$ of successes in the $j$-th slot given that there are $a=A_N^{\ssup j}$ attempts. We observe that the vector of numbers $(Z_1,\dots,Z_\kappa)$ of message transmission attempts $Z_k$ in the $k$-th channel is multinomially distributed with parameter $a=\sum_{k\in[\kappa]}Z_k$ and $\kappa$. This means, for any $\alpha\in(0,\infty)$, that
\begin{equation}\label{identidistSN}
\begin{aligned}
\P^{\ssup N}_{{\rm L,MC}}\big(S_N^{\ssup j}=s|A_N^{\ssup j}=a\big)
&=\sum_{\substack{z_1,\cdots,z_\kappa\in\N_0\colon \sum_k z_k=a \\ \sum_k \1_{\{z_k=1\}}=s}}\kappa^{-a}\binom {a}{(z_k)_k}\\
&=\frac{a!}{\kappa^\alpha}\alpha^{-a}\e^{\alpha\kappa}\sum_{\substack{z_1,\cdots,z_\kappa\in\N_0\colon \sum_k z_k=a \\ \sum_k \1_{\{z_k=1\}}=s}}\prod_{j\in[\kappa]}\left( \frac{\alpha^{z_k}}{z_k!} \e^{-\alpha}\right)\\
&=\frac{1}{\Poi_{\alpha\kappa}(a)}\Poi_{\alpha}^{\otimes \kappa}\Big(\sum_{k\in[\kappa]}X_k=a, \sum_{k\in[\kappa]}\1_{\{X_k=1\}}=s \Big),
\end{aligned}
\end{equation}
where $X_1,\dots,X_\kappa$ are independent $\Poi_\alpha$-distributed variables.  We obtain for the joint distribution of $A_N^{\ssup j}$ and $S_N^{\ssup j}$ that 
\begin{equation}
\P^{\ssup N}_{{\rm L,MC}}\big(A_N^{\ssup j}=a, S_N^{\ssup j}=s)=\frac{\Bin_{bN,p/N}(a)}{\Poi_{\alpha\kappa}(a)}\Poi_{\alpha}^{\otimes \kappa}\Big(\sum_{k\in[\kappa]}X_k=a, \sum_{k\in[\kappa]}\1_{\{X_k=1\}}=s \Big),\qquad (a,s)\in\Xi.
\end{equation}
We now pick $\alpha=bp/\kappa$ and observe that the quotient on the right-hand side then  converges towards one as $N\to\infty$, according to the Poisson limit theorem. Furthermore, the last term was introduced in \eqref{Rdef} under the name $M(a,s)$. Hence, the pair $(A_N,S_N)$ is equal to the sum of $N$ independent copies of a pair with distribution $M_N$ that converges pointwise towards $M$ as $N\to\infty$. Analogously to the corresponding part in Section~\ref{sec-ProofLDPmicroslot} (see around \eqref{expequiv}), one shows that $\frac 1N(\widetilde A_N,\widetilde S_N)$ and $\frac 1N(A_N,S_N)$ are exponentially equivalent, where the former is $\frac 1N $ times a sum of  $N$ independent random vectors $(A^{\ssup 1},S^{\ssup 1}),\dots,(A^{\ssup N},S^{\ssup N})$ with distribution $M$ each. Hence both satisfy the same LDP, if any of them satisfies some.

Indeed, $\frac 1N(\widetilde A_N,\widetilde S_N)$ does satisfy the LDP of Theorem \ref{thm-LDPMultichannel} under Rule (L), as is seen in the same way as in Section~\ref{sec-ProofLDPmicroslot}. One uses that the empirical measure $\widetilde \mu_N=\frac 1N\sum_{j=1}^N\delta_{(A^{\ssup j},S^{\ssup j})}$ satisfies an LDP with rate function $\mu\mapsto  H(\mu|M)$ and that $\frac 1N(\widetilde A_N,\widetilde S_N)=\sum_{(i,j)\in\Xi}\widetilde\mu_N(i,j) (i,j)$ is a function of $\widetilde \mu_N$ that is, after applying some cutting procedure, continuous. Then the contraction principle implies that  $\frac 1N(\widetilde A_N,\widetilde S_N)$ satisfies the LDP of Theorem \ref{thm-LDPMultichannel} under Rule (L).

\subsection{Proof of Theorem \ref{thm-LDPMultichannel} under Rule (G)}\label{sec-ProofLDPmicroslot-GMC}

In this section, we prove the LDP for $\frac 1N(A_n,S_N)$ in Scenario (MC) under Rule (G). We are able to use the identification of their distribution from Section~\ref{sec-ProofLDPmicroslot-LMC} here for a different choice of parameters. Indeed, recall that $A_N$ is $\Bin_{bN,p}$-distributed. Given that $A_N=a$ attempts are made during the entire time interval $[0,1]$, each of the $a$ attempts makes a random and uniform choice among  $N$ time slots and $\kappa$ channels altogether. Furthermore, in each channel in each slot, the success criterion is that no more than one choice is made here. This means that the distribution of $S_N$ given $\{A_N=a \}$ is the same as in \eqref{identidistSN} with $\kappa N$ instead of $\kappa$. Again, we choose $\alpha=bp/\kappa$. Hence, for any $(a,s)\in\Xi$, 
	\begin{equation}\label{distidentMCG}
	\begin{aligned}
\P_{\rm {G,IB}}^{\ssup N}\big(A_N= a,S_N=s\big)=\frac{\Bin_{bN,p}(a)}{\Poi_{bpN}(a)}\Poi_{bp/\kappa}^{\otimes \kappa N}\Big( \sum_{i=1}^{\kappa N}X_i=a,\sum_{i=1}^{\kappa N}\1_{\{X_i=1\}}=s\Big).
\end{aligned}
	\end{equation}
We use this now for $(a,s)$ replaced by $(a_N N,s_N N)\in\N^2$ with $a_N\to a$ and $s_N\to s$ for some $(a,s)\in\Xi$ and see that the quotient on the right-hand side behaves like
$$
\begin{aligned}
\lim_{N\to\infty}\frac 1N\log \frac{\Bin_{bN,p}(a_NN)}{\Poi_{bpN}(a_NN)}
&=\lim_{N\to\infty}\frac 1N\log \frac{(bN/\e)^{bN}p^{aN}(1-p)^{(b-a)N}(aN)! \e^{bpN}}
{(aN)! ((b-a)N/\e)^{(b-a)N}(bpN)^{aN}}\\
&=-\Big[a-bp+(b-a)\log\frac{1-\frac ab}{1-p}\Big],
\end{aligned}
$$
using also Stirling's formula.	

The second term on the right-hand side of \eqref{distidentMCG} is the dsitribution of the sum  of $(X_i,\1_{\{X_i=1\}})$ of $\kappa N$ independent, $\Poi_{bp/\kappa}$-distributed random variables $X_1,\dots,X_{\kappa N}$. This is a two-dimensional functional of their empirical measure $\mu_{\kappa N}$, and the latter satisfies an LDP with speed $\kappa N$ with rate function equal to $\mu\mapsto H(\mu|\Poi_{bp/\kappa})$. This functional is not a continuous one, since the identity map is not bounded, but in Section~\ref{sec-ProofLDPmicroslot} (see \eqref{cutting}) we saw how to perform a suitable cutting argument. Hence, we know that the pair $\frac 1{\kappa N}\sum_{i=1}^{\kappa N} (X_i,\1_{\{X_i=1\}})=(\langle \mu_{\kappa N},{\rm id}\rangle, \langle \mu_{\kappa N},\delta_{\{1\}}\rangle)$ satisfies, according to the contraction principle, an LDP with speed $N$ with rate function
$$
\mu\mapsto\kappa\inf\Big\{H(\mu|{\rm Poi}_{bp/\kappa})\colon\mu\in\Mcal_1(\N_0),\sum_{g\in\N_0}\mu(g)g=\frac{a}{\kappa},\mu(\{1\})=\frac{s}{\kappa} \Big\}.
$$
(The prefactor $\kappa$ comes from the change of scales from $\kappa N$ to $N$ in the LDP, and the $\kappa$ in the two denominators comes from the normalization of $\sum_{i=1}^{\kappa N}$ by $\kappa N$ instead of $N$.) Hence summarizing everything together ends the proof of Theorem~\ref{thm-LDPMultichannel} under Rule (G).

\section{Optimizing and conditioning}\label{sec-Ratefunctions}

\noindent In this section we prove Lemma~\ref{lem-optimalp} and Theorem \ref{thm-conditionalthroughput}.

\subsection{Optimizing $p\mapsto s_p$}\label{sec-optimisep}

In this section, we prove Lemma~\ref{lem-optimalp}, that is, we analyse the maximizer of the map $(0,\infty)\ni p\mapsto s_{\rm IB}(p,\kappa)$, the optimal throughput for Scenario (IB) under Rule (L). We abbreviate $s_p=s_{\rm IB}(p,\kappa)$.

The analytic function $g(a)=s_{a/b}=a \e^{-a} \sum_{i=0}^{\kappa-1}\frac{a^i}{i!}$ is positive in $(0,\infty)$ with limits $0$ at $a\downarrow 0$ and $a\to\infty$, hence it has at least one maximizer $a_*$, which is characterised by $g'(a_*)=0$. We see that (with $f_\leq(a)=\sum_{i=0}^\kappa\frac{a^i}{i!}$)
\begin{equation}\label{spderivative}
\frac{\d}{\d p} s_p=b\, \e^{-a_p}\big(f_{\leq}'(a_p)+a_p f_{\leq}''(a_p)-a_p f_\leq'(a_p)\big)=b\e^{-bp}\Big[\sum_{i\leq \kappa-1}\frac{(bp)^i}{i!}-\frac{(bp)^\kappa}{(\kappa-1)!}\Big],\qquad p>0.
\end{equation}
Hence, \eqref{pMax} characterizes the minimizer(s) $p_*$, but at this stage we do not yet know how many minimizers exist.

Using elementary calculus, we see that a solution $a_*$ to \eqref{pMax} exists since the polynomial $f(a)=-(\kappa-1)!b \e^{-a}\frac\d{\d p}s_p= a^\kappa-\sum_{i\leq\kappa-1}a^i \frac{(\kappa-1)!}{i!}$ starts with $f(0)<0$ and satisfies $f(a)\to\infty$ as $a\to\infty$. Note that, for any $a>0$, we have  
$$
\begin{aligned}
f(a)&\geq a^\kappa-\sum_{i\leq\kappa-1}a^i(\kappa-1)^{\kappa-1-i}=a^\kappa-(\kappa-1)^{\kappa-1}\sum_{i\leq\kappa-1}\Big(\frac a{\kappa-1}\Big)^i=a^\kappa+\frac{(\kappa-1)^{\kappa}-a^\kappa}{a-(\kappa- 1)}\\
&=\frac{a^\kappa(a-\kappa)+(\kappa-1)^{\kappa}}{a-(\kappa- 1)},
\end{aligned}
$$
and the latter is positive for any $a>\kappa-1$. Hence, we even have that $a_*\leq \kappa-1$. Furthermore, there is only one solution, since $f'(a)=\kappa a^{\kappa-1}-\sum_{i\leq \kappa -1}a^i \frac{(\kappa-1)!}{i!}+a^{\kappa-1}$ for any $a$, and for any solution $a_*$ we see that $f'(a_*)=(\kappa+1) a_*^{\kappa -1}-a_*^\kappa=a_*^{\kappa -1}[\kappa+1-a_*]$, which is positive. Hence, $f$ has precisely one zero in $[0,\infty)$. It is negative left of $a_*$ and positive right of it. Accordingly, $p\mapsto s_p$ is increasing in $[0, p_*]$ and decreasing in $[p_*,\infty)$. We obtain a lower bound for $a_*$ by 
$$
f(a)\leq a^\kappa -a^i\frac{(\kappa-1)!}{i!}< a^i\Big(a^{\kappa-i}-(i+1)^{\kappa-i-1}\Big),\qquad a>0, i\in\{0,\dots,\kappa-1\}.
$$
This upper bound is zero for $a=(i+1)^{1-1/(\kappa-i)}$, hence $a_*\geq \max_{i=0}^{\kappa-1}(i+1)^{1-1/(\kappa-i)}$. Taking $i=\kappa-\sqrt\kappa$ gives $a_*\geq (\kappa-\sqrt\kappa)^{1-\kappa^{-1/2}}
=\kappa(1+o(1))$ as $\kappa\to\infty$. This finishes the proof of Lemma~\ref{lem-optimalp}.

\subsection{Conditioning on successes}\label{sec-fewsuccesses}

In this section, we prove Theorem \ref{thm-conditionalthroughput}. Recall that we conceive the maximal throughput per micro slot, $s=s_p$, as a function of $p$. Recall from Lemma~\ref{lem-optimalp} that the maximal $p^*$ for $p\mapsto s_p$ is characterized by 
\begin{equation}\label{p*chracterisation}
\frac {a_p^\kappa}{(\kappa-1)!}=\sum_{i=0}^{\kappa-1}\frac{a_p^i}{i!},\qquad a_p=bp.
\end{equation}
Furthermore recall that $a_p(s)$ denotes the minimising $a$ for the map $a\mapsto \inf_r I^{\ssup p}_{\rm L,IB}(a,s,r)$, and note that $a_p=a_p(s_p)=bp$. Here we answer the question of the reason for few number of successes. The following lemma implies Theorem \ref{thm-conditionalthroughput}.

\begin{lemma}\label{lemm3} For any $p\in(0,\infty)$, we have $a_p'(s_p)<0$ for $p<p_*$ and $a_p'(s_p)>0$ for $p>p_*$. In particular, for $s$ in a neighbourhood of $s_p$, \eqref{psmall} and  \eqref{plarge} hold.

	Furthermore, for $p=p^*$, we have $a_{p^\ast}(s)>a_{p^\ast}(s_{p^\ast})=b  p_*$ for any $s\in[0,b]\setminus \{bp_*\}$.

\end{lemma}

\begin{proof} Let us first analyse $\inf_r I^{\ssup p}_{\rm L,IB}(a,s,r)$ for fixed $a, s\in(0,\infty)$ satisfying $a> s$. We benefit from the representation in \eqref{Idef}:
We have that 
$$
\begin{aligned}
\inf_r I^{\ssup p}_{\rm L,IB}(a,s,r)&=\inf_{r}\inf\{H(\mu|{\rm Poi}_{pb})\colon \langle f,\mu\rangle=(a,s,r) \}\\
&=\inf\Big\{H(\mu|{\rm Poi}_{pb})\colon \sum_{k=0}^{\infty}k\mu_k=a, \sum_{k=0}^{\kappa}k\mu_k=s \Big\}\\
&=\inf\Big\{H(\mu|{\rm Poi}_{pb})\colon \langle \mu, {\rm id}\rangle=a, \langle \mu, {\rm id}|_{\le\kappa}\rangle=s \Big\},
\end{aligned}
$$
where ${\rm id}$ is the identity function on $\N_0$ and ${\rm id}|_{\leq\kappa}(k)=k\1_{[0,\kappa]}(k)$; and we used the notation $\langle \mu, f\rangle$ for the integral of a function $f$ with respect to a measure $\mu$. Now we apply standard variational calculus. Consider a minimizer $\mu$ of the last formula. A standard argument shows that $\mu_k>0$ for any $k$. Fix some compactly supported $\gamma\colon \N_0\to\R$ satisfying $\gamma\bot\1, \gamma\bot \id$ and $\gamma\bot \id|_{\le\kappa}$. Then, for any $\eps\in\R$ with sufficiently small $|\eps|$, the measure $\mu+\eps\gamma$ is admissible. From minimality, we deduce that
$$
0=\partial_\eps|_{\eps=0}H(\mu+\eps\gamma|{\rm Poi}_{pb})=\sum_{k}\Big(\gamma_k\log\frac{\mu_k}{q_k}+\mu_k\frac{\gamma_k}{\mu_k}\Big)=\Big\langle \gamma,\log\frac{\mu}{q}\Big\rangle,
$$
where we put $q_k={\rm Poi}_{pb}(k)$. Hence, $\log \frac\mu q$ is a linear combination of $\1$, $\id$ and $\id|_{\le\kappa}$. That is, there are $A,B,C\in\R$ such that
\begin{equation}\label{min2}
	\mu_k=q_k\e^A\e^{Bk}\times
	\begin{cases} 
	\e^{Ck} &\mbox{for } k\leq \kappa, \\
	1&\mbox{for } k>\kappa,
	\end{cases}\qquad k\in \N_0.
\end{equation}
We note that $A, B$ and $C$ are well-defined functions of $a$ and $s$, since $\1$, $\id$ and $\id|_{\le\kappa}$ are linearly independent.

Now using that $\langle \mu,\1\rangle=1$ and  $\langle \mu, \id\rangle=a$ and $\langle \mu, \id|_{\le\kappa}\rangle=s$, and introducing the notation
\begin{equation}\label{varphidef}
\varphi(B,C):=\log\Big(\sum_{k=0}^{\kappa}q_k\e^{(B+C)k}+\sum_{k>\kappa}q_k\e^{Bk}\Big),\qquad B,C\in\R,
\end{equation}
we see that $B=B(a,s)$ and $C=C(a,s)$ are characterised by
\begin{eqnarray}
a&=&\frac{\underset{k\le\kappa}{\sum}kq_k\e^{(B+C)k}+\underset{k>\kappa}{\sum}kq_k\e^{Bk}}{\underset{k\le\kappa}{\sum}q_k\e^{(B+C)k}+\underset{k>\kappa}{\sum}q_k\e^{Bk}}=\partial_B \varphi(B,C),\label{achar}\\
s&=&\frac{\underset{k\le\kappa}{\sum}kq_k\e^{(B+C)k}}{\underset{k\le\kappa}{\sum}q_k\e^{(B+C)k}+\underset{k>\kappa}{\sum}q_k\e^{Bk}}=\partial_C \varphi(B,C),\label{schar}
\end{eqnarray}
while $A(a,s)=-\varphi(B(a,s),C(a,s))$. Furthermore,
\begin{equation}\label{Iminrident}
\inf_{r} {I}^{\ssup p}_{\rm L,IB}(a,s,r)=\sum_{k}\mu_k\log\frac{\mu_k}{q_k}=Ba+Cs-\varphi(B,C).
\end{equation}
This finishes the characterisation of $\inf_{r} {I}^{\ssup p}_{\rm L,IB}(a,s,r)$ for any fixed $a,s$.

Now we optimise over $a$ with $s$ fixed. We recall that $a_p(s)$ denotes the minimizing $a$ of $\inf_{r} {I}^{\ssup p}_{\rm L,IB}(a,s,r)$. Recalling that $B$ and $C$ are functions of $a$ and $s$, we differentiate \eqref{Iminrident} with respect to $a$ and use it for $a=a_p(s)$ to obtain
\begin{equation}\label{aderivative}
\begin{aligned}
0&=\Big(a_p(s)-\partial_B\varphi(B(a_p(s),s),C(a_p(s),s))\Big)\frac{\d}{\d s}B(a_p(s),s)\\
&\quad +\Big(s-\partial_C\varphi(B(a_p(s),s),C(a_p(s),s))\Big)\frac{\d}{\d s} C(a_p(s),s)+B(a_p(s),s)\\
&=B(a_p(s),s),
\end{aligned}
\end{equation}
also using \eqref{achar} and \eqref{schar}. Differentiating this with respect to $s$ produces
\begin{equation}\label{a1}
a_p'(s)=-\frac{\partial_s B(a_p(s),s)}{\partial_a B(a_p(s),s)}.
\end{equation}
A tedious calculation, starting from differentiating both \eqref{achar} and \eqref{schar} both with respect to $a$ and to $s$, gives, for $B=B(a,s)$ and any $a$ and $s$,
$$
	\partial_aB=\frac{\partial^2_C\varphi}{\partial^2_B\varphi\partial^2_C\varphi-(\partial_C\partial_B\varphi)^2}\qquad \mbox{and}\qquad 
\partial_sB=-\,\frac{\partial_B\partial_C\varphi}{\partial^2_B\varphi\partial^2_C\varphi-(\partial_C\partial_B\varphi)^2}
$$
and hence
\begin{equation}\label{Min5}
a_p'(s)=\frac{\partial_B\partial_C\varphi(B,C)}{\partial_C^2\varphi(B,C)}\qquad\mbox{with }B=B(a_p(s),s)=0\mbox{ and } C=C(a_p(s),s).
\end{equation}
First we show that the denominator is positive:
$$
\begin{aligned}
\partial_C^2\varphi(B,C)&=\frac{\underset{k\le\kappa}{\sum}k^2 q_k\e^{(B+C)k}\big(\underset{k\le\kappa}{\sum}q_k\e^{(B+C)k}+\underset{k>\kappa}{\sum}q_k\e^{Bk}\big)-\big(\underset{k\le\kappa}{\sum}kq_k\e^{(B+C)k}\big)^2}
{\big(\underset{k\le\kappa}{\sum}q_k\e^{(B+C)k}+\underset{k>\kappa}{\sum}q_k\e^{Bk}\big)^2}\\
&\geq \frac{\big(\underset{k\le\kappa}\sum k^2 q_k\e^{(B+C)k}\big)\big(\underset{k\le\kappa}{\sum} q_k\e^{(B+C)k}\big)-\big(\underset{k\le\kappa}{\sum}kq_k\e^{(B+C)k}\big)^2}
{\big(\underset{k\le\kappa}{\sum}q_k\e^{(B+C)k}+\underset{k>\kappa}{\sum}q_k\e^{Bk}\big)^2}>0,\qquad B,C\in\R,
\end{aligned}
$$
as a standard symmetrisation shows. Next we consider the  numerator in \eqref{Min5}:
\begin{equation}\label{expression}
\begin{aligned}
\partial_B\partial_C\varphi(0,C)&=\Big(\sum_{k\le\kappa}q_k\e^{Ck}+\sum_{k>\kappa}q_k\Big)^{-2}\\
&\qquad
\Big[\sum_{k\le\kappa}k^2q_k\e^{Ck}\Big(\sum_{k\le\kappa}q_k\e^{Ck}+\sum_{k>\kappa} q_k\Big)-\Big(\sum_ {k\le\kappa}kq_k\e^{Ck}+\sum_{k>\kappa}kq_k\Big)\sum_{k\le\kappa} kq_k\e^{Ck}\Big].
\end{aligned}
\end{equation}
No we use the facts that $\sum_{k\le\kappa}q_k+\sum_{k>\kappa}q_k=1$ (since $(q_k)_{k\in\N_0}$ is a probability distribution) and $\sum_{k\in\N_0}k q_k=bp=a_p=a_p(s_p)$ (see Corollary \ref{cor-LLNSALOHA}; $(q_k)_{k\in\N_0}=\Poi_{pb}$ has expectation $pb$). Furthermore, note that $C(a_p(s_p),s_p)=0$ by optimality (which can be seen in the same way as the fact that $B(a_p(s),s)=0$ above). Then we get
$$
\begin{aligned}
a_p'(s_p)&=\partial_B\partial_C\varphi(0,0)=\sum_{k\le\kappa}k^2q_k-bp\sum_{k\leq \kappa}kq_k\\
&=bp\e^{-bp}\Big[\sum_{k\leq \kappa-1}(k+1)\frac{(bp)^k}{k!}-bp\sum_{k\leq \kappa-1}\frac{(bp)^k}{k!}\Big]
=p\frac{\d}{\d p}s_p,
\end{aligned}
$$
as we see from \eqref{spderivative}. Recall that $p_*$ is the unique maximizer for $p\mapsto s_p$. According to Lemma \ref{lem-optimalp}, this (and therefore $a_p'(s_p)$)  is positive if $p<p^*$ and negative if $p>p^*$. This implies all assertions of Lemma \ref{lemm3} for $p\not=p^*$.

Now we consider the case $p=p^*$ characterised in \eqref{p*chracterisation}. Here it will not be successful to rely on the characterisation of $a_p(s)$ by $0=B(a_p(s),s)$ and to consider the derivative with respect to $s$ in $s=s_{p_*}$ only, since $\partial_B\partial_C\varphi(0,0)=0$ for $p=p^*$. Instead, we use \eqref{achar} and explicitly look at the difference
	\begin{equation}\label{a2}
	\begin{aligned}
		a_{p^\ast}(s)-a_{p^\ast}(s_{p^\ast})&=
		\partial_B \varphi(0,C)- bp^*
	=\frac{\sum_{k\le\kappa}q_k\e^{Ck}[k-a_{p^\ast}]+\sum_{k>\kappa}q_k[k-a_{p^\ast}]}{\sum_{k\le\kappa}q_k\e^{Ck}+\sum_{k>\kappa}q_k}\\
	&=\frac{\sum_{k\le\kappa}q_k[\e^{Ck}-1][k-a_{p^\ast}]}{\sum_{k\le\kappa}q_k\e^{Ck}+\sum_{k>\kappa}q_k},
\end{aligned}
\end{equation}
with $C=C(a_{p^*}(s),s)$. We used in the last step that $\sum_{k>\kappa}kq_k=a_{p^*}-\sum_{k\le\kappa}kq_k$ and $\sum_{k\in\N_0}q_k=1$. Note that $C<0$ for $s<s_{p^*}$ and $C>0$ for $s>s_{p^*}$. Indeed, a similar calculation as in \eqref{aderivative} shows that 
$$
\frac{\d}{\d s}\inf_{r,a}I^{\ssup{p^*}}_{\rm L,IB}(a,s,r)=\frac{\d}{\d s}\Big[ s C(a_{p^*}(s),s)-\varphi\big(0,C(a_{p^*}(s),s)\big)\Big]=C(a_{p^*}(s),s),\qquad s\in(0,\infty).
$$
Now note that $s_{p^*}$ is defined as the minimizer of the function $s\mapsto \inf_{r,a}I^{\ssup{p^*}}_{\rm L,IB}(a,s,r)$; hence it is decreasing left of the minimal point and increasing right of it. 

Write $g(C)=\sum_{k=0}^\kappa q_k[\e^{Ck}-1][k-a_{p^\ast}]$ for the numerator of the right-hand side of \eqref{a2}. Clearly $g(0)=0$. Recall that $\partial_B\partial_C\varphi(0,0)=0$ hence the derivative of \ref{a2} with respect to $C$ is $0$. Clearly the derivative of \eqref{a2} is $0$ only if $g'(0)=0$. Hence observe that, for any $C<0$,
$$
\begin{aligned}
g'(C)&= \sum_{k=0}^\kappa k q_k \e^{Ck} (k-a_{p_*})< \e^{Ca_{p_*}} \sum_{k\leq a_{p_*}}k q_k(k-a_{p_*})+\e^{Ca_{p_*}}\sum_{k\colon a_{p_*}<k\leq \kappa} k q_k (k-a_{p_*})=0.
\end{aligned}
$$
Hence, $g$ is strictly decreasing in $(-\infty,0]$ and hence positive in $(-\infty,0)$. An analogous argument shows that $g'(C)>0$ for $C>0$: 
$$
g'(C)= \sum_{k=0}^\kappa k q_k \e^{Ck} (k-a_{p_*})> \sum_{k\leq a_{p_*}}k q_k(k-a_{p_*})+\sum_{k\colon a_{p_*}<k\leq \kappa} k q_k (k-a_{p_*})=0.
$$
Hence $g$ is strictly increasing and positive in $(0,\infty)$. This implies that $a_{p_*}(s)>a_{p_*}(s_{p_*})$ for any $s\not=s_{p_*}$ and finishes the proof of the lemma.
\end{proof}

\bigskip

\noindent{\bf Acknowledgment.} The support of the Deutsche Akademische Auslandsdienst (DAAD) via the Project {\em Berlin-AIMS Network in Stochastic Analysis} (Project-ID 57417853) is gratefully acknowledged.

\end{document}